\documentclass{elsarticle}

\usepackage{amsmath}
\usepackage{amssymb}
\usepackage{amsthm}

\usepackage{times}

\newtheorem{lemma}{Lemma}[section]
\newtheorem{theorem}[lemma]{Theorem}
\newtheorem{definition}[lemma]{Definition}

\author{Henry Towsner}
\ead{hpt@math.ucla.edu}
\title{A Combinatorial Proof of the Dense Hindman's Theorem}
\address{Department of Mathematics, UCLA}
\date{\today}

\begin{document}

\begin{abstract}
  The Dense Hindman's Theorem states that, in any finite coloring of the natural numbers, one may find a single color and a ``dense'' set $B_1$, for each $b_1\in B_1$ a ``dense'' set $B_2^{b_1}$ (depending on $b_1$), for each $b_2\in B_2^{b_1}$ a ``dense'' set $B_3^{b_1,b_2}$ (depending on $b_1,b_2$), and so on, such that for any such sequence of $b_i$, all finite sums belong to the chosen color.  (Here density is often taken to be ``piecewise syndetic'', but the proof is unchanged for any notion of density satisfying certain properties.)  This theorem is an example of a combinatorial statement for which the only known proof requires the use of ultrafilters or a similar infinitary formalism.  Here we give a direct combinatorial proof of the theorem.
\end{abstract}


\maketitle

\section{Introduction}
Hindman's Theorem states that, in any finite coloring of the natural numbers, some color contains an infinite set and the sums of all non-empty finite subsets.  Hindman's original proof \cite{hindman74} is quite complicated; fortunately, there are both simpler combinatorial arguments \cite{baumgartner74,towsner12:_simpl_proof_some_diffic_examp_hindm_theor} and an elegant proof based on the topology of ultrafilters (see, for instance, \cite{comfort77}).

Strikingly, the ultrafilter argument gives, sometimes with little additional work, various strengthenings of the theorem for which combinatorial proofs are either much harder, or not known to exist.  (\cite{hindman98} gives a thorough exploration of many uses of ultrafilters in this context.)  One such strengthening is the following:
\begin{theorem}
  Let $\mathbb{N}=A_1\cup\cdots\cup A_r$.  There are some $i\leq r$ and a collection $\mathcal{T}$ of finite sets of natural numbers such that:
  \begin{itemize}
  \item $\emptyset\in\mathcal{T}$
  \item If $F\in\mathcal{T}$, $\{n\mid F\cup\{n\}\in\mathcal{T}\}$ is ``dense''
  \item If $F\in\mathcal{T}$ then for each non-empty finite $S\subseteq F$, $\sum_{n\in S}n\in A_i$
  \end{itemize}
  \label{promise}
\end{theorem}
Here ``dense'' can be any property satisfying certain conditions which will be described below.

In this paper we give the first combinatorial proof of this theorem, modeled on Baumgartner's proof of the ordinary Hindman's Theorem.  The key idea is the use of approximate ultrafilters, as introduced by Hirst \cite{hirst04}---countable collections of sets of natural numbers which nonetheless contain enough information to complete the proof.  The proof here is modeled on our related proof of the ordinary Hindman's Theorem \cite{towsner:MR2791353}\footnote{Indeed, we originally found a proof quite similar to that one, and only subsequently found the proof in the style of Baumgartner which we present here.  This proof is slightly more elegant, and we hope that it will shed some light on the relationship between Baumgartner's proof of the ordinary Hindman's Theorem and the ultrafilter proof.}.

The reverse mathematical strength of even the ordinary Hindman's Theorem is open; bounds are given in \cite{blass87}, and the gap between the lower and upper bounds on reverse mathematical strength there has not been improved.  The proof given here is entirely within the bounds of second order arithmetic, but well above their upper bounds; no lower bound for the Dense Hindman's Theorem is known besides the obvious one, that any lower bound for the ordinary Hindman's Theorem must also bound the Dense Hindman's Theorem.

We thank Mathias Beiglb\"ock for bringing this question to our attention, and for many discussions about the mathematics around Hindman's Theorem.  We also thank the referees for their many helpful suggestions.

\section{General Definitions}
In this paper we will work with the natural numbers $\mathbb{N}$, which we will understand as including $0$.  However everything we say would work equally well with any countable abelian semigroup with identity.

Throughout this paper, variables denoted by lowercase letters will typically be natural numbers.  Variables denoted by upppercase letters will be subsets of $\mathbb{N}$, and variables denoted by calligraphic letters ($\mathcal{U}$, $\mathcal{F}$, etc.) will be sets of sets from $\mathbb{N}$.  In fact, since all sets of sets appearing in this paper are countable, it would cause no harm to code them using sets of natural numbers.  The one exception is the property $\mathfrak{P}$, which represents the set of sets of natural numbers satisfying some shift-invariant divisible property, such as the infinite sets, the piecewise syndetic sets, or the sets of positive upper Banach density.

We write $X-n$ for $\{x\mid x+n\in X\}$.

\begin{definition}
  Let $\mathfrak{P}$ be a collection of sets from $\mathbb{N}$ such that\footnote{To keep our promise that the proof goes through in second order arithmetic, we should insist that $\mathfrak{P}$ be given by some arithmetic formula; this includes all the examples given.}:
  \begin{itemize}
  \item $\mathbb{N}\in\mathfrak{P}$
  \item $\emptyset\not\in\mathfrak{P}$
  \item If $X\subseteq Y$ and $X\in\mathfrak{P}$ then $Y\in\mathfrak{P}$ (\emph{upwards closure})
  \item If $X_0\cup X_1=X$ and $X\in\mathfrak{P}$ then either $X_0\in\mathfrak{P}$ or $X_1\in\mathfrak{P}$ (\emph{partition regularity})
  \item For any $X$ and any $n$, $X\in\mathfrak{P}$ iff $X-n\in\mathfrak{P}$ (\emph{shift invariance})
  \end{itemize}

\end{definition}
Properties satisfying all but the shift invariance condition are called \emph{divisible} \cite{glasner80}.

Two natural examples of such properties $\mathfrak{P}$ are:
\begin{itemize}
\item $\mathfrak{P}$ is the collection of infinite sets
\item $\mathfrak{P}$ is the collection of sets $X$ such that $\sum_{x\in X}1/x=\infty$
\end{itemize}

A more interesting example is piecewise syndeticity:
\begin{definition}
 $X$ is \emph{piecewise syndetic} if there is an $n$ such that for every $m$, there is an interval $I$ with $|I|>m$ so that for each $x\in I$, $[x,x+n]\cap X$ is non-empty.
\end{definition}
See \cite{bergelson:MR1743799} for various properties of piecewise syndetic sets.

Another interesting example is positive upper Banach density:
\begin{definition}
$X$ has \emph{positive upper Banach density} if there are an $\epsilon>0$ and, for every $n$, there is an interval $I$ with $|I|>n$ such that $\frac{|I\cap X|}{|I|}>\epsilon$.
\end{definition}
See \cite{bergelson:MR1411215,jin:MR1869316,jin:MR1866042} for various properties of sets with positive upper Banach density.

The collection of piecewise syndetic sets and the collection of sets of positive upper Banach density are both valid choices for the collection $\mathfrak{P}$.

\begin{definition}
  We say a collection $\mathcal{U}$ of sets of natural numbers has the $\mathfrak{P}$-finite intersection property ($\mathfrak{P}$-fip) if for every finite collection $\mathcal{F}\subseteq\mathcal{U}$, 
\[\bigcap_{S\in \mathcal{F}}S\in\mathfrak{P}.\]

Let $\mathcal{U}$ be a countable collection of sets of natural numbers\footnote{None of our arguments would change if uncountable collections---say, true ultrafilters---are allowed.  However we wish to emphasize that none of our arguments will require more than countable collections}.  We write $\mathcal{U}^{fil}$ for the filter generated by $\mathcal{U}$, so $X\in\mathcal{U}^{fil}$ if there is a finite $\mathcal{F}\subseteq\mathcal{U}$ such that $\bigcap_{S\in\mathcal{F}}S\subseteq X$.
  
  We say $\mathcal{U}$ is a $\mathfrak{P}$-semigroup if $\mathcal{U}$ satisfies $\mathfrak{P}$-fip and whenever $X\in\mathcal{U}$, there is a $Y\in\mathcal{U}^{fil}$ such that $X-n\in\mathcal{U}^{fil}$ for each $n\in Y$.
\end{definition}

We avoid equating $\mathcal{U}$ with $\mathcal{U}^{fil}$ to emphasize that we will only concern ourselves with countably generated filters.  We now show that this causes no harm, since the properties of $\mathcal{U}$ will dictate appropriate properties for $\mathcal{U}^{fil}$.

\begin{lemma}
If $\mathcal{U}$ is a $\mathfrak{P}$-semigroup then so is $\mathcal{U}^{fil}$. 
\label{fil_also_fip}
\end{lemma}
\begin{proof}
In light of the preceding lemma, it suffices to show that whenever $X\in\mathcal{U}^{fil}$, there is a $Y\in\mathcal{U}^{fil}$ such that $X-n\in\mathcal{U}^{fil}$ for each $n\in Y$.  Let $X\in\mathcal{U}^{fil}$, and choose $\mathcal{F}\subseteq\mathcal{U}$ finite so that $\bigcap_{S\in\mathcal{F}}S\subseteq X$.  Then for each $S\in\mathcal{F}$, there is a $Y_S\in\mathcal{U}^{fil}$ so that for each $n\in Y_S$, $S-n\in\mathcal{U}^{fil}$.  Let $Y=\bigcap_{S\in\mathcal{F}}Y_S\in\mathcal{U}^{fil}$, so for each $n\in Y$, we have $S-n\in\mathcal{U}^{fil}$ for each $S\in\mathcal{F}$, and therefore $(\bigcap_{S\in\mathcal{F}}(S-n))=(\bigcap_{S\in\mathcal{F}}S)-n\subseteq X-n\in\mathcal{U}^{fil}$.     
\end{proof}

There are two useful ways to view $\mathfrak{P}$-semigroups.  The first is to observe that every $\mathfrak{P}$-semigroup represents a closed semigroup in the Stone-\v{C}ech compactification of the the discrete topology on $\mathbb{N}$ (or, equivalently, in the space of ultrafilters on $\mathbb{N}$): the semigroup corresponding to $\mathcal{U}$ is $\bigcap_{S\in\mathcal{U}}\overline{S}$ (where by $\overline{S}$, we mean the closure of the set $S$ in the Stone-\v{C}ech topology).  (\cite{hindman98} is a thorough reference on this topic.)  In particular, the proof of Hindman's Theorem using the Stone-\v{C}ech compactification makes use of the existence of idempotents; using the axiom of choice, every $\mathfrak{P}$-semigroup can be refined to an idempotent (indeed, to an idempotent consisting only of sets from $\mathfrak{P}$).

The second is to recall that an IP set is a set $S$ such that there is an infinite $T\subseteq S$ all of whose finite sums also belong to $S$.
\begin{definition}
If $T$ is a subset of $\mathbb{N}$, define
\[FS(T)=\{\sum_{i\in F}i\mid F\subseteq S, F\text{ finite and non-empty}\}.\]
\end{definition}
Then $S$ is an IP set if there is an infinite set $T$ with $FS(T)\subseteq S$.  The collection $\{FS(T)-n\mid n\in FS(T)\}$ (where $FS(T)$ is the finite sums from $T$) is a canonical example of a $\mathfrak{Q}$-semigroup where $\mathfrak{Q}$ is the collection of infinite sets.  The notion of a $\mathfrak{P}$-semigroup generalizes an IP set in two directions: first, it allows for more general choices of $\mathfrak{P}$.  Second, if we have an infinite descending sequence of IP sets $S_1\supseteq S_2\supseteq\cdots S_n\supseteq\cdots$, their intersection may well be $\emptyset$.  However the union of the corresponding $\mathfrak{P}$-semigroups is still a $\mathfrak{P}$-semigroup.  So $\mathfrak{P}$-semigroups also generalize IP sets by accommodating the result of infinitely many successive refinements of an IP set.

Indeed, this relationship reverses: it is not hard to see that if $S$ belongs to a $\mathfrak{P}$-semigroup then $S$ is an IP set.  Indeed, if we could prove that, in every partition of $\mathbb{N}$, one element of the partition belonged to a $\mathfrak{P}$-semigroup, we would be finished.  However, while this is certainly true (using arguments about the Stone-C\v{e}ch compactification), we are not aware of a direct combinatorial proof, so our ultimate argument will be less direct.

The argument here is very similar to a proof based on the Stone-C\v{e}ch compactification, but we emphasize that the $\mathfrak{P}$-semigroups appearing in our proof are much simpler objects: they are countable collections (for instance, they can be coded using only sets of natural numbers), built with no use of the axiom of choice.

We first prove some basic properties about $\mathfrak{P}$-semigroups.
\begin{lemma}
If $\mathcal{U}_0\subseteq\mathcal{U}_1\subseteq\cdots$ are $\mathfrak{P}$-semigroups then so is $\mathcal{U}=\bigcup_{n\in\mathbb{N}}\mathcal{U}_n$. 
\end{lemma}
\begin{proof}
Let $\mathcal{F}\subseteq\mathcal{U}$ be finite; then there is some $n$ such that $\mathcal{F}\subseteq\mathcal{U}_n$, and since $\mathcal{U}_n$ is a $\mathfrak{P}$-semigroup, $\bigcap_{S\in\mathcal{F}}S\in\mathfrak{P}$.

Let $X\in\mathcal{U}$.  Then $X\in\mathcal{U}_n$ for some $n$, so there is a $Y\in\mathcal{U}_n^{fil}\subseteq\mathcal{U}^{fil}$ such that for each $m\in Y$, $X-m\in\mathcal{U}_n^{fil}\subseteq\mathcal{U}^{fil}$. 
\end{proof}

\begin{lemma} 
If $\mathcal{U}$ is a $\mathfrak{P}$-semigroup and $\mathcal{U}\cup\{S-n\mid n\in S\}$ satisfies $\mathfrak{P}$-fip then $\mathcal{U}\cup\{S-n\mid n\in S\}$ is a $\mathfrak{P}$-semigroup.
\label{shifts_also_semigroup}
\end{lemma}
\begin{proof}
It suffices to check the semigroup property.  We claim that for each $m\in S-n$, $(S-n)-m\in\mathcal{U}\cup\{S-n\mid n\in S\}$.  This follows since if $m\in S-n$ then $n+m\in S$, so $(S-n)-m=S-(n+m)\in\mathcal{U}\cup\{S-n\mid n\in S\}$.
\end{proof}

\begin{lemma}
  Let $\mathcal{U}$ satisfy $\mathfrak{P}$-fip, and let $A$ be a set of natural numbers.  Then either $\mathcal{U}\cup\{A\}$ or $\mathcal{U}\cup\{\mathbb{N}\setminus A\}$ satisfies $\mathfrak{P}$-fip.
\label{one_or_the_other}
\end{lemma}
\begin{proof}
  Suppose neither collection satisfies $\mathfrak{P}$-fip.  Then choose finite sets $\mathcal{F},\mathcal{F}'\subseteq \mathcal{U}$ such that
$(\bigcap_{S\in\mathcal{F}}S\cap A)\not\in\mathfrak{P}$
and
$(\bigcap_{S\in\mathcal{F}'}S\cap (\mathbb{N}\setminus A))\not\in\mathfrak{P}$.
Then
\[\bigcap_{S\in\mathcal{F}\cup\mathcal{F'}}S\subseteq (\bigcap_{S\in\mathcal{F}}S\cap A)\cup(\bigcap_{S\in\mathcal{F}'}S\cap (\mathbb{N}\setminus A)).\]
But this is impossible, since $\bigcap_{S\in\mathcal{F}\cup\mathcal{F'}}S\in\mathfrak{P}$ must hold.
\end{proof}

\section{Dense Hindman's Theorem}
\begin{lemma}
Let $G\subseteq\mathbb{N}$ be a finite set, and suppose $Z\cap\bigcap_{m\in G}X_m\in\mathfrak{P}$ and for each $m\in G$, $X_m\cap Y_m\not\in\mathfrak{P}$.  Then $Z\cap\bigcap_{m\in G}(\mathbb{N}\setminus Y_m)\in\mathfrak{P}$. 
\label{splits_pieces}
\end{lemma}
\begin{proof}
Note that
\[Z\cap\bigcap_{m\in G}X_m\subseteq\bigcup_{m\in G}(X_m\cap Y_m)\cup(Z\cap\bigcap_{m\in G}(\mathbb{N}\setminus Y_m)).\]
Since for each $m\in G$, $X_m\cap Y_m\not\in\mathfrak{P}$, also $\bigcup_{m\in G}(X_m\cap Y_m)\not\in\mathfrak{P}$.  Therefore $\bigcup_{m\in G}(Z\cap\bigcap_{m\in G}(\mathbb{N}\setminus Y_m))\in\mathfrak{P}$ as desired.
\end{proof}

\begin{lemma}
Let $\mathcal{U}$ be a $\mathfrak{P}$-semigroup, let $X\in\mathcal{U}^{fil}$, and let $Y=\{n\mid X-n\in\mathcal{U}^{fil}\}$.  Then for every $n\in Y$, $Y-n\in\mathcal{U}^{fil}$. 
\label{shift_inv}
\end{lemma}
\begin{proof}
Let $n\in Y$.  Since $X-n\in\mathcal{U}^{fil}$, there is a $Z\in\mathcal{U}^{fil}$ such that for each $m\in Z$, $(X-n)-m=X-(n+m)\in\mathcal{U}^{fil}$.  Therefore $Z\subseteq Y-n$, and since $Z\in\mathcal{U}^{fil}$, also $Y-n\in\mathcal{U}^{fil}$. 
\end{proof}

\begin{lemma}
Let $\mathcal{U}$ be a $\mathfrak{P}$-semigroup and let $A\subseteq\mathbb{N}$ be such that $\mathcal{U}\cup\{ A\}$ does not satisfy $\mathfrak{P}$-fip.  Then there is a $\mathfrak{P}$-semigroup $\mathcal{V}$ extending $\mathcal{U}$ such that $\mathbb{N}\setminus A\in\mathcal{V}$.
\label{build_semigroup}
\end{lemma}
\begin{proof}
Let $X\in\mathcal{U}^{fil}$ be such that $X\cap A\not\in\mathfrak{P}$.  Let $n\in Y$ iff $X-n\in\mathcal{U}^{fil}$; clearly $0\in Y$.  Set $\mathcal{V}=\mathcal{U}\cup\{(\mathbb{N}\setminus A-n)\mid n\in Y\}$.

We first claim that $\mathcal{V}$ satisfies $\mathfrak{P}$-fip.  Let $Z\in\mathcal{U}^{fil}$ and let $G\subseteq Y$ be finite.  Since $Z\cap\bigcap_{n\in G}X-n\in\mathfrak{P}$ and for each $n$, $(X-n)\cap (A-n)\not\in\mathfrak{P}$, by Lemma \ref{splits_pieces}, $Z\cap\bigcap_{n\in G}(\mathbb{N}\setminus A-n)\in\mathfrak{P}$.

We now show that $\mathcal{V}$ satisfies the semigroup property.  Let $n\in Y$; we must find a witness for the set $\mathbb{N}\setminus A-n$.  By the previous lemma, we have $Y-n\in\mathcal{U}^{fil}$, and we claim that for each $m\in Y-n$, $(\mathbb{N}\setminus A-n)-m\in\mathcal{V}$.  Since $m\in Y-n$, $n+m\in Y$, we have $(\mathbb{N}\setminus A)-(n+m)=(\mathbb{N}\setminus A-n)-m\in\mathcal{V}$.  
\end{proof}

\begin{lemma}
  If $\mathcal{U}$ is a $\mathfrak{P}$-semigroup, $A\subseteq\mathbb{N}$, $S\subseteq\mathbb{N}$ has the property that
\[\mathcal{U}\cup\{\mathbb{N}\setminus A-n\mid n\not\in S\}\]
satisfies $\mathfrak{P}$-fip, and $\mathcal{U}\cup\{S-n\mid n\in S\}$ does not satisfy $\mathfrak{P}$-fip, then there are a finite $F\subseteq S$ and a $Y\in\mathcal{U}^{fil}$ such that $0\in Y$ and $\mathcal{U}\cup\{\bigcup_{n\in F}(\mathbb{N}\setminus A-n-m)\mid m\in Y\}$ is a $\mathfrak{P}$-semigroup.
\label{3.1}
\end{lemma}
\begin{proof}
  Since $\mathcal{U}\cup\{S-n\mid n\in S\}$ does not satisfy $\mathfrak{P}$-fip, let $X\in\mathcal{U}^{fil}$ and $F\subseteq S$ be finite such that $X\cap\bigcap_{n\in F}(S-n)\not\in\mathfrak{P}$.  Let $Y=\{m\mid X-m\in\mathcal{U}^{fil}\}$; clearly $0\in Y$ since $X\in\mathcal{U}^{fil}$.  Note that for any $m$, $(X-m)\cap\bigcap_{n\in F}(S-n-m)\not\in\mathfrak{P}$.

We claim that $\mathcal{U}\cup\{\bigcup_{n\in F}(\mathbb{N}\setminus A-n-m)\mid m\in Y\}$ satisfies $\mathfrak{P}$-fip.  Let $Z\in\mathcal{U}^{fil}$ and $G\subseteq Y$ be finite.  Let $Z'\in\mathcal{U}^{fil}$ be such that $Z-n\in\mathcal{U}^{fil}$ for each $n\in Z'$.  By Lemma \ref{fil_also_fip}, $Z'\in\mathfrak{P}$.  Since $Z'\cap \bigcap_{m\in G}(X-m)\in\mathcal{U}^{fil}\subseteq\mathfrak{P}$, we may apply Lemma \ref{splits_pieces} to obtain
\[Z'\cap\bigcap_{m\in G}\bigcup_{n\in F}(\mathbb{N}\setminus S-n-m)\in\mathfrak{P}.\]
In particular, this set is non-empty, so it contains some element $k$.  Therefore $Z-k\in\mathcal{U}^{fil}$ and for each $m\in G$ there is an $n\in F$ such that $k+n+m\not\in S$, and therefore
\[Z-k\cap\bigcap_{m\in G}\bigcup_{n\in F}(\mathbb{N}\setminus A-n-m-k)\in\mathfrak{P}\]
since $\mathcal{U}\cup\{\mathbb{N}\setminus A-n\mid n\not\in S\}$ satisfies $\mathfrak{P}$-fip.  Since $\mathfrak{P}$ is shift invariant, we also have
\[Z\cap\bigcap_{m\in G}\bigcup_{n\in F}(\mathbb{N}\setminus A-n-m)\in\mathfrak{P}\]
as desired.

To see that $\mathcal{U}\cup\{\bigcup_{n\in F}(\mathbb{N}\setminus A-n-m)\mid m\in Y\}$ satisfies the semigroup property, we proceed as in the previous lemma: for each $m\in Y$, $Y-m\in\mathcal{U}^{fil}$ by Lemma \ref{shift_inv}, and for each $k\in Y-m$, $m+k\in Y$ and therefore $(\bigcup_{n\in F}(\mathbb{N}\setminus A-n-m))-k=\bigcup_{n\in F}(\mathbb{N}\setminus A-n-(m+k))\in\mathcal{U}\cup\{\bigcup_{n\in F}(\mathbb{N}\setminus A-n-m)\mid m\in Y\}$.
\end{proof}

%
%
%

\begin{definition}
  We say $A$ is \emph{large} relative to a $\mathfrak{P}$-semigroup $\mathcal{U}$ if whenever $\mathcal{V}$ is a $\mathfrak{P}$-semigroup extending $\mathcal{U}$, $\mathcal{V}\cup\{A\}$ satisfies $\mathfrak{P}$-fip.
\end{definition}

A consequence of this definition is the following:
\begin{lemma}
For any $A$ and any $\mathfrak{P}$-semigroup $\mathcal{U}$, either $A$ is large for $\mathcal{U}$ or there is a $\mathfrak{P}$-semigroup $\mathcal{V}\supseteq\mathcal{U}$ such that $(\mathbb{N}\setminus A)\in\mathcal{V}^{fil}$.
\label{choice}
\end{lemma}
\begin{proof}
This follows by applying Lemma \ref{build_semigroup} to the definition.
\end{proof}

\begin{lemma}
  If $C=C_1\cup\cdots\cup C_n$ and $C$ is large relative to $\mathcal{U}$, there are a $\mathfrak{P}$-semigroup $\mathcal{V}$ extending $\mathcal{U}$ and an $i$ such that $C_i$ is large for $\mathcal{V}$.
\end{lemma}
\begin{proof}
  Applying induction, it suffices to consider the case $n=2$.  If $C_1$ is large relative to $\mathcal{U}$ then $1$ and $\mathcal{U}$ suffice.  Otherwise there is a $\mathcal{V}$ extending $\mathcal{U}$ such that $C_2\in\mathcal{V}^{fil}$, and so certainly $C_2$ is large relative to $\mathcal{V}$.
\end{proof}

\begin{theorem}
  Let $\mathcal{U}$ be a $\mathfrak{P}$-semigroup and $A\subseteq\mathbb{N}$.  If $A$ is large for $\mathcal{U}$ then there is a $\mathfrak{P}$-semigroup $\mathcal{V}$ extending $\mathcal{U}$ such that $\{n\in A\mid A\cap (A-n)\text{ is large for }\mathcal{V}\}\in\mathfrak{P}$.
\end{theorem}
\begin{proof}
Fix an enumeration $F_0,\ldots,F_n,\ldots$ of the finite non-empty sets of natural numbers.  We construct a sequence of $\mathfrak{P}$-semigroups $\mathcal{U}=\mathcal{U}_0\subseteq\mathcal{U}_1\subseteq\cdots$ such that for each $n$ and all $i< n$, either $\bigcap_{m\in F_i}(A-m)$ is large for $\mathcal{U}_n$ or $\bigcup_{m\in F_i}(\mathbb{N}\setminus(A-m))\in\mathcal{U}^{fil}_n$.

Suppose we have constructed $\mathcal{U}_n$.  If $\bigcap_{m\in F_{n}}(A-m)$ is large for $\mathcal{U}_n$, we take $\mathcal{U}_{n+1}=\mathcal{U}_n$.  Otherwise we apply Lemma \ref{choice} to obtain $\mathcal{U}_{n+1}\supseteq\mathcal{U}_n$ so that $\mathcal{U}_{n+1}^{fil}$ contains $\bigcup_{m\in F_{n}}(\mathbb{N}\setminus (A-m))$.

We then set $\mathcal{V}=\bigcup_n\mathcal{U}_n$.  $\mathcal{V}$ is a $\mathfrak{P}$-semigroup, and for any finite non-empty set $F$, either $\bigcap_{m\in F}(A-m)$ is large for $\mathcal{V}$ or $\bigcup_{m\in F}(\mathbb{N}\setminus(A-m))\in\mathcal{V}^{fil}$.

We now construct a set $S$ inductively, setting $S_0=\{0\}$ and $S_{n+1}=S_n$ if $\bigcap_{i\in S_n}(A-i)\cap(A-(n+1))$ is not large for $\mathcal{V}$ and $S_{n+1}=S_n\cup\{n+1\}$ if $\bigcap_{i\in S_n}(A-i)\cap(A-(n+1))$ is large for $\mathcal{V}$.  Set $S=\bigcup_n S_n$.  $S$ has the property that for each finite non-empty $F\subseteq S$, $\bigcap_{n\in F}(A-n)$ is large for $\mathcal{V}$, while if $m\not\in S$ then there is a finite non-empty $F\subseteq S$ (indeed, $\{n\in S\mid n<m\}$) such that $\bigcap_{n\in F\cup\{m\}}(A-n)$ is not large for $\mathcal{V}$.

Let $F\subseteq\mathbb{N}\setminus S$ be finite and non-empty.  Then we may choose a $G\subseteq S$ so that for each $m\in F$, $\bigcap_{n\in G\cup\{m\}}(A-n)$ is not large for $\mathcal{V}$.  Then, by the construction of $\mathcal{V}$, $\bigcup_{n\in G\cup\{m\}}(\mathbb{N}\setminus A-n)\in\mathcal{V}^{fil}$ for each $m\in F$, and since $\mathcal{V}\cup\{A-n\mid n\in G\}$ satisfies $\mathfrak{P}$-fip, in particular $\bigcap_{m\in F}\bigcup_{n\in G\cup\{m\}}(\mathbb{N}\setminus A-n)\cap\bigcap_{n\in G}A-n\in\mathfrak{P}$.  But this set is simply $\bigcap_{m\in F}\mathbb{N}\setminus A-m$, so $\mathcal{V}\cup\{\mathbb{N}\setminus A-m\mid m\in F\}$ satisfies $\mathfrak{P}$-fip.  This holds for every finite $F\subseteq \mathbb{N}\setminus S$, so $\mathcal{V}\cup\{\mathbb{N}\setminus A-m\mid m\not\in S\}$ satisfies $\mathfrak{P}$-fip.

Therefore, by Lemma \ref{3.1}, if $\mathcal{V}\cup\{S-n\mid n\in S\}$ does not satisfy $\mathfrak{P}$-fip then there are a finite set $F\subseteq S$ and a $\mathfrak{P}$-semigroup $\mathcal{W}$ extending $\mathcal{V}$ such that $\bigcup_{n\in F}(\mathbb{N}\setminus A-n)\in\mathcal{W}^{fil}$.  But this would contradict the fact that $\bigcap_{n\in F}A-n$ is large for $\mathcal{V}$.  So by Lemma \ref{shifts_also_semigroup}, $\mathcal{V}\cup\{S-n\mid n\in S\}$ satisfies $\mathfrak{P}$-fip, and is therefore a $\mathfrak{P}$-semigroup.

Since $A$ is large for $\mathcal{V}$, also $\mathcal{V}\cup\{S-n\mid n\in S\}\cup\{A\}$ satisfies $\mathfrak{P}$-fip, so in particular, $S\cap A\in\mathfrak{P}$.  Since for each $n\in S$, $A\cap (A-n)$ is large for $\mathcal{V}$, the claim is proven.
\end{proof}

\begin{theorem}
  Let $\mathbb{N}=A_1\cup\cdots\cup A_r$.  There are some $i\leq r$ and a collection $\mathcal{T}$ of finite sets of natural numbers such that:
  \begin{itemize}
  \item $\emptyset\in\mathcal{T}$
  \item If $F\in\mathcal{T}$, $\{n\mid F\cup\{n\}\in\mathcal{T}\}$ belongs to $\mathfrak{P}$
  \item If $F\in\mathcal{T}$ then $FS(F)\subseteq A_i$
  \end{itemize}
\end{theorem}
\begin{proof}
  Since $\mathbb{N}$ is large for the trivial $\mathfrak{P}$-semigroup $\{\mathbb{N}\}$, we may choose a $\mathfrak{P}$-semigroup $\mathcal{U}$ and an $i$ so that $A_i$ is large for $\mathcal{U}$.  We will now place sets $F$ in $\mathcal{T}$, ensuring that whenever $F\in\mathcal{T}$, we have a $\mathfrak{P}$-semigroup $\mathcal{U}_F$ such that $\bigcap_{n\in FS(F)\cup\{0\}}A_i-n$ is large for $\mathcal{U}_F$.
  
  We start by placing $\emptyset$ in $\mathcal{T}$, and we have $\mathcal{U}_\emptyset=\mathcal{U}$.  Now suppose $F\in\mathcal{T}$ and set $A'=\bigcap_{n\in FS(F)\cup\{0\}}A_i-n$.  By the preceding theorem, there is a $\mathfrak{P}$-semigroup $\mathcal{U}'$ extending $\mathcal{U}_F$ such that $\{n\in A'\mid A'\cap (A'-n)\text{ is large for }\mathcal{U}'\}$ belongs to $\mathfrak{P}$.  We place $F\cup\{n\}$ in $\mathcal{T}$ for each such $n$, and set $\mathcal{U}_{F\cup\{n\}}=\mathcal{U}'$.
\end{proof}

As noted above, we have not used any special properties of $\mathbb{N}$ beyond being a countable abelian semigroup with identity; moreover, it is trivial to add an identity to a semigroup.  So we have:
\begin{theorem}
  Let $\mathbb{S}$ be a countable abelian semigroup and let $\mathfrak{P}$ be a shift-invariant divisible property on $\mathbb{S}$.  Let $\mathbb{S}=A_1\cup\cdots\cup A_r$.  There are an $i\leq r$ and a collection $\mathcal{T}$ of finite sets of elements of $\mathbb{S}$ such that:
  \begin{itemize}
  \item $\emptyset\in\mathcal{T}$
  \item If $F\in\mathcal{T}$, $\{s\mid F\cup\{s\}\in\mathcal{T}\}$ belongs to $\mathfrak{P}$
  \item If $F\in\mathcal{T}$ then $FS(F)\subseteq A_i$
  \end{itemize}
\end{theorem}

\bibliographystyle{plain}
\bibliography{../../Bibliographies/main}
\end{document}